\documentclass[a4paper,12pt,leqno]{amsart}
\usepackage[utf8]{inputenc}
\usepackage{amssymb,amsmath,amsthm,latexsym}
\usepackage{lmodern}

\newcommand{\bbR}{\mathbb{R}}

\newtheorem{thm}{Theorem}
\newtheorem{lem}[thm]{Lemma}
\title[A solution to the problem of the existence\ldots]{A solution to the problem of the existence of a dense plastic subset $X\subseteq\bbR$ of cardinality $|X|<\frak{c}$}
\author{Wojciech Bielas}
\newcommand{\bbQ}{\mathbb{Q}}
\newcommand{\calR}{\mathcal{R}}

\newcounter{czesc}
\renewcommand{\geq}{\geqslant}
\renewcommand{\leq}{\leqslant}

\begin{document}
\maketitle

\begin{abstract}
  Under the assumption $\frak{c}\geqslant \omega_2$,  we give an example of a dense plastic subset $X\subseteq\bbR$ of cardinality $|X|<\frak{c}$.
  This answers Problem 1 of \cite{Banakh}.
\end{abstract}

\section{Introduction}
For a metric space $(X,d)$,  we say that $f\colon X\to X$ is \emph{non-expansive} if $d(f(x),f(y))\leq d(x,y)$ for all $x,y\in X$.
A metric space $X$ is \emph{plastic} if its every non-expansive bijection 
$f\colon X\to X$ is an isometry.
A subset $Y$ of a metric space is \emph{plastic} if it is plastic as a subspace.

In \cite{Banakh} Authors have proved, among others, that every countable dense subspace of a normed space is not plastic.
This implies that, under the Continuum Hypothesis, every dense subset $X\subseteq\bbR$ of cardinality $|X|<\frak{c}$ is not plastic.
In Problem 1 of \cite{Banakh}, Authors ask whether the Continuum Hypothesis is a necessary assumption.
Under the assumption $\frak{c}\geq\omega_2$, we construct a dense plastic subset $X\subseteq\bbR$ of cardinality $|X|=\omega_1$.
The subset $X$ has a stronger property: the identity map is its only non-expansive bijection.

\section{The construction}
Let $\calR=\{[a_n,b_n]\colon n<\omega\}$ be a pairwise disjoint family of intervals of $\bbR$ such that for every $m,n<\omega$:
 \begin{enumerate}
 \item $a_n,b_n\in\bbQ$,
 \item $2(b_{n+1}-a_{n+1})<b_n-a_n$,
 \item if  $b_m<a_n$, then $a_n-b_m>b_0-a_0$,
 \item $a_0<a_1$,
 \item for every $x\in\bbR$ there exist $k$ and $l$ such that $b_k<x<a_l$.
 \end{enumerate}
 Condition (ii) says that the distance between different intervals of $\calR$ is at least $b_0-a_0$.
 It follows that for each $n$ there exists a unique $m$ such that:
 \begin{itemize}
 \item for every $x\in [a_n,b_n]$ and $y\in [a_m,b_m]$ we have $x<y$,
 \item $(b_n,a_m)\cap\bigcup\calR=\emptyset$.
 \end{itemize}
 In that case let $A_n\subseteq(b_n,a_m)$ be such that $|A_n\cap U|=\omega_1$ for every non-empty open $U\subseteq (b_n,a_m)$.
 We define
 $$X=\bigcup_{n<\omega}((\bbQ\cap[a_n,b_n])\cup A_n).$$

\section{Basic properties of the set $X$}
For the rest of this section we fix a non-expansive bijection $f\colon X\to X$ and define for every $n<\omega$:
$$\delta_n=b_n-a_n,\quad a'_n=a_n+\delta_{n+1},\quad b'_n=b_n-\delta_{n+1}.$$

Let us recall that $x\in\bbR$ is a \emph{condensation} point of a subset $A\subseteq\bbR$ if $|A\cap U|\geq\omega_1$ for every open neighborhood $U$ of $x$.
Thus, if $x$ is a condensation point of $X$, then $f(x)$ is also a condensation point of $X$.
From now on, when we say a condensation point, we will mean a condensation point of $X$.
Observe that for every $n<\omega$:
\begin{itemize}
\item $a_n$ and $b_n$ are condensation points,
\item there is no condensation point $x$ such that $|x-a'_n|<\delta_{n+1}$ or $|x-b'_n|<\delta_{n+1}$,
\item neither $f^{-1}(a'_n)$ nor $f^{-1}(b'_n)$ is a condensation point.
\end{itemize}
 \begin{lem}\label{lem1}
For every $n<\omega$,    if $f^{-1}(a'_n),f^{-1}(b'_n)\in [a'_n,b'_n],$
then
$$\{f^{-1}(a'_n),f^{-1}(b'_n)\}=\{b'_n,a'_n\}.$$
 \end{lem}

 \begin{proof}
   Points $f^{-1}(a'_n)$ and $f^{-1}(b'_n)$ belong to the interval $[a'_n,b'_n]$ of length $b_n-a_n-2\delta_{n+1}$, hence
   $$|f^{-1}(a'_n)-f^{-1}(b'_n)|\leq b_n-a_n-2\delta_{n+1}.$$
Since $f$ is non-expansive,
   $$b_n-a_n-2\delta_{n+1}=|a'_n-b'_n|\leq |f^{-1}(a'_n)-f^{-1}(b'_n)|.$$
Thus the distance between points $f^{-1}(a'_n)$ and $f^{-1}(b'_n)$ equals $b_n-a_n-2\delta_{n+1}$ but the interval $[a'_n,b'_n]$ contains only one such pair: $\{a'_n,b'_n\}$.
 \end{proof}

 \begin{lem}\label{lem2}
For every $n<\omega$,   if $f(a'_n)=a'_n$, then $f(a_n)=a_n$ and $f(b_n)=b_n$, but if
 $f(a'_n)=b'_n$, then $f(a_n)=b_n$ and $f(b_n)=a_n$.
 \end{lem}
 
 \begin{proof}
   Assume $f(a'_n)=a'_n$.
   We have $|a_n-a'_n|=\delta_{n+1}$, hence $|f(a_n)-f(a'_n)|\leq \delta_{n+1}$.
   Since $a_n$ is a condensation point, so is $f(a_n)$, but there is only one condensation point $x$ such that  $|x-a'_n|\leq\delta_{n+1}$, namely, $x=a_n$.
   Thus $f(a_n)=a_n$.
   From $f(a'_n)=a'_n$ it follows that $f(b'_n)=b'_n$ and we show similarly that $f(b_n)=b_n$.

   The case $f(a'_n)=b'_n$ is analogous.
 \end{proof}

 \begin{lem}\label{lem3}
For every $n<\omega$, if $\{f(a_n),f(b_n)\}=\{a_n,b_n\}$, then
   $f|_{[a_n,b_n]\cap X}$ is an isometry onto $[a_n,b_n]\cap X$.
 \end{lem}

 \begin{proof}
Since $\{f(a_n),f(b_n)\}=\{a_n,b_n\}$, for any $t\in [a_n,b_n]$, the system of inequalities
   $$(*)\quad\left\{\begin{array}{rcl}
  |f(a_n)-x|&\leq& |a_n-t|,\\
  |f(b_n)-x|&\leq& |b_n-t|\\
   \end{array}\right.$$
has a unique solution, which is $t$, if $f(a_n)=a_n$, and $x=a_n+b_n-t$, if $f(a_n)=b_n$.
 On the other hand, since $f$ is non-expansive, we have 
  $$(**)\quad\left\{\begin{array}{rcl}
  |f(a_n)-f(t)|&\leq& |a_n-t|,\\
  |f(b_n)-f(t)|&\leq& |b_n-t|.\\
   \end{array}\right.$$
Thus $f(t)$ is the only solution of $(*)$, hence $f(t)=t$ for all $t\in[a_n,b_n]$, if $f(a_n)=a_n$, or  $f(t)=a_n+b_n-t$ for all $t\in[a_n,b_n]$, if $f(a_n)=b_n$.
 \end{proof}

 
 \begin{lem}\label{lem4}
   For every $n<\omega$, 
$f|_{[a_n,b_n]\cap X}$ is an isometry onto $[a_n,b_n]\cap X$.
 \end{lem}

 \begin{proof}
Fix  $n$ and assume that the hypothesis holds for $k<n$.
We will show that
$$f^{-1}(a'_n),f^{-1}(b'_n)\in [a'_n,b'_n].$$
Suppose that $f^{-1}(a'_n)\notin [a'_n,b'_n]$.
By the inductive assumption, 
$$f[[a_k,b_k]\cap X]=[a_k,b_k]\cap X$$
for $k<n$.
Since $f(f^{-1}(a'_n))=a'_n\in [a_n,b_n]$, we have $f^{-1}(a'_n)\notin \bigcup_{k<n}[a_k,b_k]$.
Thus either $f^{-1}(a'_n)\in X\setminus \bigcup_{k<\omega}[a_k,b_k]$, which means that $f^{-1}(a'_n)$ is a condensation point and this is a contradiction, or $f^{-1}(a'_n)\in [a_m,b_m]\cap X$ for some $m\geq n$.
Since $a_m$ is a condensation point, so is $f(a_m)$.
Now, if $m=n$, then $f^{-1}(a'_n)\notin [a'_n,b'_n]$ implies that
$$f^{-1}(a'_n)\in [a_n,a'_n)\cup (b'_n,b_n],$$
which means that
$$\min\{|f^{-1}(a'_n)-a_n|,|b_n-f^{-1}(a'_n)|\}<\delta_{n+1}.$$
If $|f^{-1}(a'_n)-a_n|<\delta_{n+1}$, then
$$|f(a_n)-a'_n|=|f(a_n)-f(f^{-1}(a'_n))|\leq |a_n-f^{-1}(a'_n)|<\delta_{n+1},$$
but there is no condensation point $x$ such that $|x-a'_n|<\delta_{n+1}$; a contradiction.
The case $|b_n-f^{-1}(a'_n)|<\delta_{n+1}$ is analogous.
If $m>n$, then
$$\textstyle\min\{|a_m-f^{-1}(a'_n)|,|b_m-f^{-1}(a'_n)|\}\leq\frac12\delta_m.$$
If $|a_m-f^{-1}(a'_n)|\leq\frac12\delta_m$, then
$$\textstyle|f(a_m)-a'_n|=|f(a_m)-f(f^{-1}(a'_n))|\leq |a_m-f^{-1}(a'_n)|\leq\frac12\delta_{m}<\delta_{n+1};$$
a contradiction.
The case $|b_m-f^{-1}(a'_n)|\leq\frac12\delta_m$ is analogous.
Finally, $f^{-1}(a'_n)\in [a'_n,b'_n]$ and the fact that $f^{-1}(b'_n)\in [a'_n,b'_n]$ is proved similarly.

From Lemma \ref{lem1} it follows that $\{f^{-1}(a'_n),f^{-1}(b'_n)\}= \{a'_n,b'_n\}$.
By Lemma \ref{lem2}, we have $\{f(a_n),f(b_n)\}=\{a_n,b_n\}$.
Then, by Lemma \ref{lem3}, $f|_{[a_n,b_n]\cap X}$ is an isometry onto $[a_n,b_n]\cap X$.
 \end{proof}

\section{The main theorem}
 
 \begin{thm}
The set  $X$ is plastic.
\end{thm}

\begin{proof}
  Fix a bijection $f\colon X\to X$ which is non-expansive.
  By Lemma \ref{lem4}, for every $n<\omega$ the restriction  $f|_{[a_n,b_n]\cap X}$ is an isometry onto $[a_n,b_n]\cap X$.
  If $f(b_0)=a_0$ and $f(b_1)=a_1$, then
  $$|f(b_1)-f(b_0)|=a_1-a_0= (a_1-b_0)+(b_0-a_0)>(a_1-b_0)+(b_1-a_1)=b_1-b_0;$$
a contradiction.
    If $f(b_0)=a_0$ and $f(b_1)=b_1$, then
  $$|f(b_1)-f(b_0)|=b_1-a_0>b_1-b_0;$$
  a contradiction.
  Thus $f(b_0)=b_0$, hence $f(a_0)=a_0$.
    Fix $1\leq n<\omega$ and assume that $f(a_k)=a_k$ and $f(b_k)=b_k$ for $k<n$.
    Suppose that $f(b_{n+1})=a_{n+1}$.
    If $a_n<a_{n+1}$, then $b_n<a_{n+1}$ and
    $$b_{n+1}-b_n=f(a_{n+1})-f(b_n)\leq a_{n+1}-b_n<b_{n+1}-b_n;$$
    a contradiction.
    If $a_{n+1}<a_n$, then $b_{n+1}<a_n$ and
    $$a_{n}-a_{n+1}=f(a_n)-f(b_{n+1})\leq a_n-b_{n+1}<a_{n+1}-a_n;$$
    a contradiction.
    Thus $f(b_{n+1})=b_{n+1}$, hence $f(a_{n+1})=a_{n+1}$.
    By induction, $f(a_n)=a_n$ and $f(b_n)=b_n$ for every $n<\omega$.

  Fix $t\in X$.
  There exist $m$ and $n$ such that $a_m<t<a_n$.
  Once again, $t$ is the only solution of the system of inequalities
     $$(***)\quad\left\{\begin{array}{rcl}
       |a_m-x|&\leq& |a_m-t|,\\
       |a_n-x|&\leq& |a_n-t|,\\
     \end{array}\right.$$
 On the other hand, since $f$ is non-expansive, we have 
  $$(**)\quad\left\{\begin{array}{rcl}
  |f(a_m)-f(t)|&\leq& |a_m-t|,\\
  |f(a_n)-f(t)|&\leq& |a_n-t|.\\
  \end{array}\right.$$
Since $f(a_m)=a_m$ and $f(a_n)=a_n$, the value $f(t)$ is the only solution of $(***)$, hence $f(t)=t$.
\end{proof}
 
\begin{thm}
The  Continuum Hypothesis is equivalent to the non-existence of a plastic dense subset  $X\subseteq\bbR$ of cardinality $|X|<\frak{c}$.
\end{thm}

\begin{proof}
 The Continuum Hypothesis and  Corollary 2 from \cite{Banakh} imply that there is no plastic dense subset $X\subseteq \bbR$ of cardinality $|X|<\frak{c}$.

If we assume that $\frak{c}\geq\omega_2$, then the set $X$ is plastic, dense and is of cardinality $\omega_1<\frak{c}$.
\end{proof}

\end{document}